\documentclass[11pt,a4paper]{amsart}
\usepackage{mathrsfs}
\usepackage{amssymb,amsmath,mathtools}
\usepackage{hyperref}
 
\DeclareMathOperator{\Tr}{tr}

\DeclareMathOperator{\diag}{diag}
\newcommand{\dd}{\mathrm{d}}
\newcommand{\bE}{\mathbb{E}}

\newcommand{\bt}{\mathbf{t}}
\newcommand{\R}{\mathbb{R}}

\numberwithin{equation}{section}
\newtheorem{theorem}{Theorem}[section]
\newtheorem{lemma}[theorem]{Lemma}
\newtheorem{proposition}[theorem]{Proposition}
\newtheorem{corollary}[theorem]{Corollary}
\newtheorem{definition}[theorem]{Definition}

\newtheorem{remark}[theorem]{Remark}
\newtheorem{example}[theorem]{Example}

\title[AOT between Gaussian processes]{Adapted optimal transport between Gaussian processes in discrete time}

\author{Madhu Gunasingam} 
\address{Department of Statistical Sciences, University of Toronto} 
\email{madhu.gunasingam@mail.utoronto.ca}

\author{Ting-Kam Leonard Wong} 
\address{Department of Statistical Sciences, University of Toronto} 
\email{tkl.wong@utoronto.ca}

\date{\today}
\keywords{Adapted optimal transport, Wasserstein distance, Bures-Wasserstein, Knothe-Rosenblatt, adapted Brenier, Gaussian process}

\begin{document}

\maketitle
\begin{abstract}
We derive explicitly the adapted $2$-Wasserstein distance between non-degenerate Gaussian distributions on $\mathbb{R}^N$ and characterize the optimal bicausal coupling(s). This leads to an adapted version of the Bures-Wasserstein distance on the space of positive definite matrices. 
\end{abstract}

\section{Introduction}
{\it Adapted} {\it optimal transport (AOT)} has emerged to be an appropriate framework for quantifying distributional uncertainty and the sensitivity of stochastic optimization problems in contexts where the flow of information in time plays a crucial role. The key idea, which has several origins in different fields, is to impose a (bi)causal constraint on the feasible couplings. We recall the necessary concepts in Section \ref{sec:AOT} and refer the reader to \cite{BBBE20, BBBW22, BBEP20, BBYZ2017, BBP2021, BPS24, EP24, lassalle2018causal, P24}, and the references therein, for detailed expositions of the theory and the related literature.

Several AOT problems have been explicitly solved. For example, \cite{BBYZ2017, BKR22, R85} provide sufficient conditions for the {\it Knothe-Rosenblatt coupling}, recalled in Section \ref{sec:KR}, to be optimal in discrete time. Also see \cite{BKR22, BT19, lassalle2018causal, RS24} for continuous-time results concerning stochastic differential equations. The main purpose of this paper is to address the adapted $2$-Wasserstein transport between arbitrary non-degenerate Gaussian measures on $\mathbb{R}^N$ which are possibly non-Markovian. To the best of our knowledge, this basic case is still open despite the rapid growth of the subject in recent years. The closest papers we could find are \cite{H23}, which studies a distributionally robust filtering problem with Gaussian noise, and \cite{Z20}, which solves a related OT problem between stationary Gaussian processes. Now we state the main theorem of the paper.

\begin{theorem} \label{thm:main}
	Let $\mu = \mathcal{N}(a, A)$ and $\nu = \mathcal{N}(b, B)$ be non-degenerate Gaussian distributions on $\mathbb{R}^N$. Let $A = LL^{\top}$ and $B = MM^{\top}$ be the Cholesky decompositions of $A$ and $B$ respectively. Then the adapted $2$-Wasserstein distance $\mathcal{AW}_2(\mu, \nu)$ is given by
	\begin{equation} \label{eqn:AW.Gaussian}
		\mathcal{AW}_2^2(\mu, \nu) = \|a - b\|^2 + d_{\mathrm{ABW}}^2(A, B),
	\end{equation}
	where $d_{\mathrm{ABW}}$ is the adapted Bures-Wasserstein distance on the space $\mathscr{S}_{++}(N)$ of (symmetric) positive definite matrices defined by
	\begin{equation} \label{eqn:ABW.Gaussian}
		d_{\mathrm{ABW}}^2(A, B) := \Tr(A) + \Tr(B) - 2 \|\diag(L^{\top}M)\|_1.
	\end{equation}
	Here $\diag(\cdot)$ gives the diagonal of a square matrix, $\|\cdot\|_1$ is the $\ell_1$-norm and $\|\cdot\| = \|\cdot\|_2$ is the Euclidean norm. (See Remark \ref{rmk:techniques}(i) for an extension to a weighted square cost.)
\end{theorem}

In Section \ref{sec:AW} we prove Theorem \ref{thm:main} 
via a {\it dynamic programming principle} from \cite{BBYZ2017} and characterize the set of optimal bicausal couplings (Corollary \ref{cor:optimal.couplings}). The main idea is to observe that the Gaussian assumption and the sum separability of the quadratic cost lead to a tractable linear-quadratic structure of the value function (see \eqref{eqn:V.reduced.form}).  Recently, the authors of \cite{AHP24} built on our techniques to extend Theorem \ref{thm:main} to multivariate Gaussian processes together with entropic regularization. We discuss some geometric properties of the adapted Bures-Wasserstein distance in Section \ref{sec:AW.distance}, and provide illustrative examples in Section \ref{sec:examples}. Since the Gaussian distribution is fundamental in various applications, our explicit solution and its extensions in \cite{AHP24} may improve the understanding of the general theory and stimulate new applications of AOT. Some directions for future research are discussed in Section \ref{sec:conclusion}.

\section{Adapted optimal transport} \label{sec:AOT}
We work with Borel probability measures on $\mathbb{R}^N$, $N \geq 1$, regarded as laws of real-valued stochastic processes indexed by $t \in [N] := \{1, \ldots, N\}$ and equipped with their natural filtrations (implicit in \eqref{eqn:causal} below). We use $\mathcal{P}(E)$ to denote the collection of Borel probability measures on a space $E$. By a {\it coupling} of $(\mu, \nu) \in (\mathcal{P}(\mathbb{R}^N))^2$ we mean an element $\pi \in \mathcal{P}(\mathbb{R}^N \times \mathbb{R}^N)$ whose first and second marginals are $\mu$ and $\nu$ respectively. The set of couplings of $(\mu, \nu)$ is denoted by $\Pi(\mu, \nu)$. By an abuse of terminology we also use the word coupling to mean a pair $(X = (X_t)_{t = 1}^N, Y = (Y_t)_{t = 1}^N)$ of real-valued processes, defined on some probability space, with joint law $\pi \in \Pi(\mu, \nu)$. We write $\mathbb{P}_{\pi}$ for a probability under which $(X, Y) \sim \pi$ (analogous for the expectation $\mathbb{E}_{\pi}$).

For integers $s \leq t$ we let $s:t = (s, s + 1, \ldots, t)$. We also write $\bt = 1:t$ and $\bt' = (t+1):N$. For a quantity $Z$ indexed by $t \in [N]$ we use the shorthand $Z_{s:t} = (Z_s, \ldots, Z_t)$ (by convention $Z_{1:0}$ and $Z_{(N+1):N}$ are empty). Suppose $\mu \in \mathcal{P}(\mathbb{R}^N)$ and $X \sim \mu$. For $t \in [N]$ we let $\mu_t(\cdot) \in \mathcal{P}(\mathbb{R})$ be the marginal distribution of $X_t$. For $x_{1:t} \in \mathbb{R}^t$, we let $\mu_{t+1}(\cdot|x_{1:t}) \in \mathcal{P}(\mathbb{R})$ be the conditional distribution of $X_{t+1}$ given $X_{1:t} = x_{1:t}$, and let $\mu_{\bt'}(\cdot|x_{\bt}) \in \mathcal{P}(\mathbb{R}^{N-t})$ be the conditional distribution of $X_{\bt'}$ given $X_{\bt} = x_{\bt}$. Analogously, we define $\nu_{t+1}(\cdot|y_{\bt})$ and $\nu_{\bt'}(\cdot|y_{\bt})$ for $Y$, as well as $\pi_{t + 1}(\cdot| x_{\bt}, y_{\bt})$ and $\pi_{{\bf t}'}(\cdot| x_{\bt}, y_{\bt})$ for $(X, Y)$. These (regular) conditional distributions are well-defined by the disintegration theorem and are a.s.~unique with respect to suitable reference measures.

Let $\mathcal{P}_2(\mathbb{R}^N) = \{\mu \in \mathcal{P}(\mathbb{R}^N): \int_{\mathbb{R}^N} \|x\|^2 \mu(\dd x) < \infty\}$ be the collection of Borel probability measures on $\mathbb{R}^N$ with finite second moment. For $\mu, \nu \in \mathcal{P}_2(\mathbb{R}^N)$, the usual {\it $2$-Wasserstein distance} $\mathcal{W}_2(\mu, \nu)$ is defined by
\begin{equation} \label{eqn:W2}
	\mathcal{W}_2^2(\mu, \nu) =  \inf_{\pi \in \Pi(\mu, \nu)} \int_{\mathbb{R}^N \times \mathbb{R}^N} \|x - y\|^2 \pi(\dd x \dd y).
\end{equation}
We simply call $\mathcal{W}_2$ the {\it Wasserstein distance} since $2$ is the only order considered in this paper (similar for other quantities such as the Knothe-Rosenblatt distance \eqref{eqn:KR.distance}). When $\mu$ is absolutely continuous, Brenier's theorem asserts that there exists a unique optimal coupling $\pi_{\mathrm{W}}^{\mu, \nu}$ given in terms of the gradient of a convex function \cite{V03}. Let $\mathscr{S}_{++}(N)$ (resp.~$\mathscr{S}_{+}(N)$) be the space of $N$-by-$N$ (symmetric) positive definite (resp.~semipositive definite) matrices. When $\mu = \mathcal{N}(a, A)$ and $\nu = \mathcal{N}(b, B)$ are Gaussian distributions on $\mathbb{R}^N$,\footnote{We write $\mathcal{N}_N(a, A)$ when there is a need to emphasize the dimension.} it is well known that $\mathcal{W}_2^2(\mu, \nu) = \|a - b\|^2 + d_{\mathrm{BW}}^2(A, B)$, where $d_{\mathrm{BW}}$ is the {\it Bures-Wasserstein distance} on $\mathscr{S}_{+}(N)$ defined by 
\begin{equation} \label{eqn:BW}
	d_{\mathrm{BW}}^2(A, B) = \Tr(A) + \Tr(B) - 2 \Tr\big(  A^{\frac{1}{2}} B  A^{\frac{1}{2}} \big)^{\frac{1}{2}}.
\end{equation}
Here $A^{\frac{1}{2}} \in \mathscr{S}_+(N)$ is the unique matrix square root of $A \in \mathscr{S}_+(N)$. See \cite{BJL19, WMP18, T11} for in-depth studies of $d_{\mathrm{BW}}$ from the viewpoints of matrix analysis and Riemannian geometry. When $A, B \in \mathscr{S}_{++}(N)$, $\pi_{\mathrm{W}}^{\mu, \nu}$ is the deterministic coupling under which $Y = T_{\mathrm{W}}^{\mu, \nu}X$, where $X \sim \mu$ and
\begin{equation} \label{eqn:W.OT.map}
	T_{\mathrm{W}}^{\mu, \nu}x = b + A^{-\frac{1}{2}} \big(  A^{\frac{1}{2}} B  A^{\frac{1}{2}} \big)^{\frac{1}{2}} A^{-\frac{1}{2}} (x - a), \quad x \in \mathbb{R}^N,
\end{equation}
is the {\it Brenier map} from $\mu$ to $\nu$. Here and below we use the column convention for $X$, $Y$, $a$, $b$, etc. When $N = 1$ this reduces to the {\it comonotonic coupling} induced by the non-decreasing transport map $F_{\nu}^{-1} \circ  F_{\mu}$, where $F$ denotes the distribution function. For later use, we also recall the {\it counter-monotonic coupling} which has the non-increasing transport map $F_{\nu}^{-1} \circ R \circ F_{\mu}$, where $R(u) = 1 - u$. Note that the Jacobian matrix in \eqref{eqn:W.OT.map} is an element of $\mathscr{S}_{++}(N)$ since $T_{\mathrm{W}}^{\mu, \nu}$ is a convex gradient. Unless it is diagonal, the transport map $T_{\mathrm{W}}^{\mu, \nu}$ ``looks into the future of $X$'' since generally each $Y_t$ is a function of both $X_{\bt}$ and $X_{\bt'}$. The causal condition excludes such couplings.

\begin{definition}[Causal and bicausal couplings] 
	\label{def:causal.coupling}
	Let $\mu, \nu \in \mathcal{P}(\mathbb{R}^N)$. A coupling $\pi \in \Pi(\mu, \nu)$ is said to be causal in the direction $\mu$ to $\nu$, if for all $t \in [N]$ and Borel $B \subset \mathbb{R}^t$ we have
	\begin{equation} \label{eqn:causal}
		\mathbb{P}_{\pi}(Y_{1:t} \in B|X_{1:N}) = \mathbb{P}_{\pi}(Y_{1:t} \in B|X_{1:t}).
	\end{equation}
	We say that $\pi$ is bicausal if it is causal in both directions, and let $\Pi_{bc}(\mu, \nu)$ be the collection of bicausal couplings of $(\mu, \nu)$.
\end{definition}

We cite an equivalent condition, taken from \cite[Proposition 5.1]{BBYZ2017}, for the bicausal property. Given $\pi \in \Pi(\mu, \nu)$,  decompose it as a product of regular conditional distributions:
\begin{equation} \label{eqn:coupling.decomp}
	\begin{split}
		\pi(\dd x \dd y) 
		= \pi_1(\dd x_1 \dd y_1) \pi_2(\dd x_2 \dd y_2|x_1, y_1) 
		\cdots \pi_N(\dd x_{N} \dd y_N|x_{1:(N-1)}, y_{1:(N-1)}).
	\end{split}
\end{equation}
Then $\pi \in \Pi_{bc}(\mu, \nu)$ if and only if $\pi_1 \in \Pi(\mu_1, \nu_1)$ and, for all $1 \leq t \leq N - 1$ and ($\pi$-almost) all $x_{1:t}, y_{1:t} \in \R^t$, we have $\pi_{t+1}(\cdot|x_{1:t}, y_{1:t}) \in \Pi(\mu_{t+1}(\cdot|x_{1:t}), \nu_{t+1}(\cdot|y_{1:t}))$. In \cite{R85} this is called a {\it Markov-construction}. If in \eqref{eqn:W2} we optimize over bicausal couplings, we obtain the {\it adapted Wasserstein distance}.

\begin{definition}[Adapted Wasserstein distance]
	For $\mu, \nu \in \mathcal{P}_2(\mathbb{R}^N)$, the adapted Wasserstein distancec $\mathcal{AW}_2(\mu, \nu)$ is defined by
	\begin{equation} \label{eqn:bicausal.optimal.transport}
		\mathcal{AW}_2^2(\mu, \nu) = \inf_{\pi \in \Pi_{bc}(\mu, \nu)} \int_{\mathbb{R}^N \times \mathbb{R}^N} \|x - y\|^2 \pi(\dd x \dd  y).
	\end{equation}
\end{definition}         

Clearly $\mathcal{AW}_2 \geq \mathcal{W}_2$, and equality holds when $N = 1$. One can show that $\mathcal{AW}_2$ is a metric on $\mathcal{P}_2(\mathbb{R}^N)$; its induced topology is the fundamental {\it weak adapted topology} which has several equivalent characterizations \cite{BBBE20, BLO23, P24}. In this paper we focus on the case where $\mu$ and $\nu$ are Gaussian, i.e., $X$ and $Y$ are Gaussian processes. 



\section{The Knothe-Rosenblatt coupling} \label{sec:KR}
For $\mu, \nu \in \mathcal{P}(\mathbb{R}^N)$, the {\it Knothe-Rosenblatt coupling} $\pi_{\mathrm{KR}}^{\mu, \nu}$, also called the {\it multivariate quantile transform}, defines a bicausal coupling which is known to be $\mathcal{AW}_2$-optimal under a stochastic co-monotonicity condition on $\mu$ and $\nu$ \cite[Proposition 3.4]{BKR22} (also see Remark \ref{rmk:comonotone} below). See \cite{AGK23, BMZ23} for recent applications of Knothe-Rosenblatt transport maps (also called {\it monotone triangular transport maps}) in statistics and machine learning. 

We consider directly the Gaussian case where $\mu = \mathcal{N}(a, A)$, $\nu = \mathcal{N}(b, B)$ with $A, B \in \mathscr{S}_{++}(N)$. It is useful to parameterize $A$ and $B$ in terms of their {\it Cholesky decompositions}. Let $\mathscr{L}_{++}(N)$ (resp.~$\mathscr{L}_+(N)$) be the space of $N$-by-$N$ lower triangular matrices with positive (resp.~non-negative) diagonal entries. Then there exist unique $L, M \in \mathscr{L}_{++}(N)$ such that $A = LL^{\top}$ and $B = MM^{\top}$. It is not difficult to show (see, e.g., \cite[Proposition 2]{L19}) that the Cholesky map $A \mapsto L$ is a diffeomorphism from $\mathscr{S}_{++}(N)$ onto $\mathscr{L}_{++}(N)$. When $A \in \mathscr{S}_+(N)$ there still exists $L \in \mathscr{L}_+(N)$ such that $A = LL^{\top}$, but $L$ is not necessarily unique (an example is given in \eqref{eqn:Cholesky.not.unique}). On some probability space, let $\epsilon \sim \mathcal{N}_N(0, I)$ be an $N$-dimensional standard Gaussian random vector. The {\it Knothe-Rosenblatt coupling} $\pi_{\mathrm{KR}}^{\mu, \nu}$ between $\mu$ and $\nu$ is given by the law of $(X, Y) = (a + L\epsilon, b + M\epsilon)$, and is characterized by the property that each $(\pi_{\mathrm{KR}}^{\mu, \nu})_{t+1}(\cdot|x_{\bt}, y_{\bt})$ is the comonotonic coupling between the univariate Gaussian distributions $\mu_{t+1}(\cdot|x_{\bt})$ and $\nu_{t+1}(\cdot|y_{\bt})$. We may also call $\pi_{\mathrm{KR}}^{\mu, \nu}$ the {\it synchronous coupling} since the same noise sequence $\epsilon$ is used to drive both $X$ and $Y$ as solutions to suitable stochastic difference equations. 

For univariate processes, the Knothe-Rosenblatt coupling coincides with what we call the {\it adapted Brenier coupling} $\pi_{\mathrm{AB}}^{\mu, \nu}$ defined via \eqref{eqn:coupling.decomp} by
\begin{equation} \label{eqn:adapted.Brenier}
	(\pi_{\mathrm{AB}}^{\mu, \nu})_{t+1}(\cdot|x_{\bt}, y_{\bt}) = \pi_{\mathrm{W}}^{\mu_{t+1}(\cdot|x_{\bt}), \nu_{t+1}(\cdot|y_{\bt})}.
\end{equation}
That is, the conditional marginals are coupled by the Brenier map which is $\mathcal{W}_2$-optimal. While the Knothe-Rosenblatt coupling does not have a direct analogue for multivariate processes, the property \eqref{eqn:adapted.Brenier} makes sense in any dimension. In particular, if we consider (absolutely continuous) $\mu, \nu \in \mathcal{P}_2(\mathbb{R}^{d \times N})$ where $d$ is the space dimension and $N$ is the time dimension, then $\pi_{\mathrm{AB}}^{\mu, \nu} = \pi_{\mathrm{KR}}^{\mu, \nu}$ when $d = 1$ and $\pi_{\mathrm{AB}}^{\mu, \nu} = \pi_{\mathrm{W}}^{\mu, \nu}$ when $N = 1$ (in this case $\mathcal{W}_2 = \mathcal{AW}_2$). In this paper we use the terminology of Knothe-Rosenblatt coupling but note that the adapted Brenier coupling is more relevant in the study of multivariate processes.\footnote{This possibility is briefly discussed in \cite[Section 5]{BBYZ2017}.}

Returning to the univariate Gaussian setting, we observe that $\pi_{\mathrm{KR}}^{\mu, \nu}$ is a deterministic coupling under which
\begin{equation} \label{eqn:map.KR}
	Y = T_{\mathrm{KR}}^{\mu, \nu}X, \quad \text{where} \quad T_{\mathrm{KR}}^{\mu, \nu}x = b + ML^{-1}(x - a).
\end{equation}
Since $L, M \in \mathscr{L}_{++}(N)$, we have $ML^{-1} \in \mathscr{L}_{++}(N)$. Following \cite{BPP23}, we define the {\it Knothe-Rosenblatt distance} $\mathcal{KR}_2(\mu, \nu)$ between $\mu$ and $\nu$ by
\begin{equation} \label{eqn:KR.distance}
	\mathcal{KR}_2^2(\mu, \nu) = \int_{\mathbb{R}^N \times \mathbb{R}^N} \|x - y\|^2 \pi_{\mathrm{KR}}^{\mu, \nu}(\dd x \dd y).
\end{equation}
Since the Knothe-Rosenblatt coupling is bicausal, from \eqref{eqn:bicausal.optimal.transport} we have $\mathcal{KR}_2 \geq \mathcal{AW}_2$. The following lemma is a special case of Lemma \ref{lem:correlated.coupling} below.

\begin{lemma}
	For non-degenerate $\mu = \mathcal{N}(a, A = LL^{\top})$ and $\nu = \mathcal{N}(b, B = MM^{\top})$, we have $\mathcal{KR}_2^2(\mu, \nu) = \|a - b\|^2 + d_{\mathrm{KR}}^2(A, B)$, where $d_{\mathrm{KR}}$ defined by
	\begin{equation} \label{eqn:d.KR}
		d_{\mathrm{KR}}^2(A, B) = \Tr(A) + \Tr(B) - 2\Tr(L^{\top}M) = \|L - M\|_{\mathrm{F}}^2
	\end{equation}
	is the Knothe-Rosenblatt distance and $\|L\|_{\mathrm{F}} = \sqrt{\Tr(L^{\top}L)}$ is the Frobenius norm.
\end{lemma}

\begin{corollary} \label{cor:KR.geometry}
	The Cholesky map $A \mapsto L$ is an isometry between the metric spaces  $(\mathscr{S}_{++}(N), d_{\mathrm{KR}})$ and $(\mathscr{L}_{++}(N), \|\cdot\|_{\mathrm{F}})$.
\end{corollary}

The next lemma (inspired by the {\it correlated Brownian motion} in \cite{BKR22, BT19}) provides a potential improvement of the Knothe-Rosenblatt coupling and motivates the expression of $d_{\mathrm{ABW}}$ in \eqref{eqn:ABW.Gaussian}. We denote by $\mathscr{P}(N)$ the set of $N$-by-$N$ diagonal matrices $P = (\delta_{st} \rho_t)$ ($\delta_{st}$ is the Kronecker delta) where $\rho_t \in [-1, 1]$ for all $t \in [N]$.

\begin{lemma} \label{lem:correlated.coupling}
	Let $\mu = \mathcal{N}(a, A = LL^{\top})$ and $\nu = \mathcal{N}(b, B = MM^{\top})$ be non-degenerate. For $P \in \mathscr{P}(N)$, let $\pi^P$ be the law of $(X, Y)$ where
	\begin{equation} \label{eqn:pi.P}
		\begin{bmatrix} X \\ Y \end{bmatrix} \sim \mathcal{N}_{2N}\left(\begin{bmatrix} a \\ b \end{bmatrix}, \begin{bmatrix} A & LPM^{\top} \\ MPL^{\top} & B \end{bmatrix} \right).
	\end{equation}
	Then $\pi^{P} \in \Pi_{bc}(\mu, \nu)$. We have $\min_{P \in \mathcal{P}} \bE_{\pi^P}[\|X - Y\|^2] = \|a - b\|^2 + d_{\mathrm{ABW}}^2(A, B)$, where $d_{\mathrm{ABW}}$ is defined by \eqref{eqn:ABW.Gaussian}, and the minimum is attained by $P = (\delta_{st} \rho_t)$ if and only if
	\begin{equation} \label{eqn:rho.sign}
		\rho_t =
		\left\{\begin{array}{ll}
			+1, & \text{if } (L^{\top} M)_{tt} > 0,\\
			-1, & \text{if } (L^{\top} M)_{tt} < 0.
		\end{array}\right.
	\end{equation}
\end{lemma}
\begin{proof}
	Observe that $\pi^P$ is the law of $(X, Y) = (a + L\epsilon^X, b + M\epsilon^Y)$, where $\epsilon^X = \epsilon_{1:N}^X$ and $\epsilon^Y = \epsilon_{1:N}^Y$ are random vectors with
	\begin{equation} \label{eqn:correlated.noise}
		\begin{bmatrix} \epsilon^X \\ \epsilon^Y \end{bmatrix} \sim \mathcal{N}_{2N}\left(\begin{bmatrix} 0 \\ 0 \end{bmatrix}, \begin{bmatrix} I & P \\ P & I \end{bmatrix} \right).
	\end{equation}
	Since $L, M \in \mathscr{L}_{++}(N)$, from the equivalent condition below Definition \ref{def:causal.coupling} we see that $\pi$ is a bicausal coupling of $(\mu, \nu)$ (we recover $\pi_{\mathrm{KR}}^{\mu, \nu}$ when $P = I$). Using \eqref{eqn:correlated.noise}, the transport cost under $\pi^P$ is given by
	\begin{equation*} 
		\begin{split}
			&\bE_{\pi^P}\left[\|X - Y\|^2\right] = \bE_{\pi^P} \left[ \|(a + L \epsilon^X) - (b + M \epsilon^Y)\|^2 \right] \\
			&= \|a - b\|^2 + \bE_{\pi^P}\left[ (\epsilon^X)^{\top} L^{\top} L \epsilon^X + (\epsilon^Y)^{\top} M^{\top} M \epsilon^Y - (\epsilon^X)^{\top} L^{\top} M \epsilon^Y - (\epsilon^Y)^{\top} M^{\top}L \epsilon^X \right]  \\
			&= \|a - b\|^2 + \Tr(A) + \Tr(B) - 2 \sum_{t = 1}^N \rho_t (L^{\top}M)_{tt}.
		\end{split}
	\end{equation*}
	It is clear that to minimize this quantity we should pick $\rho_t$ to match the sign of $(L^{\top} M)_{tt}$. This yields \eqref{eqn:rho.sign} and the desired minimum value.
\end{proof}

In Lemma \ref{lem:correlated.coupling} the correlations $\rho_t$ are deterministic. In principle, we may consider more general bicausal couplings where $\rho_t$ is previsible, i.e., $\rho_t = \rho_t(X_{1:(t-1)}, Y_{1:(t-1)})$. It turns out that this is not necessary in our discrete time Gaussian setting. We also observe that when $(L^{\top} M)_{tt} = 0$ the transport cost  $\bE_{\pi^P}\left[\|X - Y\|^2\right]$  does not depend on $\rho_t$. This suggests that the $\mathcal{AW}_2$-optimal coupling is not unique when $\diag(L^{\top} M)$ contains a zero entry. An explicit example is given in Section \ref{sec:examples}.

\section{Dynamic programming principle} \label{sec:AW}
In this section we prove Theorem \ref{thm:main} and characterize in Corollary \ref{cor:optimal.couplings} the optimal bicausal coupling(s). Our main tool is a {\it dynamic programming principle (DPP)} for bicausal optimal transport which we state only for $\mathcal{AW}_2$.

\begin{theorem}[Dynamic programming principle (Proposition 5.2 of \cite{BBYZ2017}] \label{thm:DPP}
	Given $\mu, \nu \in \mathcal{P}_2(\mathbb{R}^N)$, define the value functions $V_t : \mathbb{R}^t \times \mathbb{R}^t \rightarrow [0, \infty)$, $t = 0, \ldots, N$ (when $t = 0$, $V_0 \in \mathbb{R}$ is a constant), by $V_N(x_{1:N}, y_{1:N}) = \|x_{1:N} - y_{1:N}\|^2$ and, for $t= 0, 1, \ldots, N - 1$,
	\begin{equation} \label{eqn:DPP}
		\begin{split}
			V_t(x_{1:t}, y_{1:t}) = \inf_{\pi_{t+1}} \int_{\R \times \R} V_{t+1}(x_{1:(t+1)}, y_{1:(t+1)}) \pi_{t+1}(\dd x_{t + 1} \dd y_{t+1}),
		\end{split}
	\end{equation}
	where the infimum is over $\pi_{t+1} \in \Pi(\mu_{t+1}(\cdot|x_{\bt}), \nu_{t+1}(\cdot|y_{\bt})) $. Then $V_0 = \mathcal{AW}_2^2(\mu, \nu)$.
\end{theorem}

The DPP reduces the AOT problem \eqref{eqn:bicausal.optimal.transport} to a sequence of one-dimensional OT problems between the conditional marginals. While the DPP is complicated for generic distributions, the Gaussian assumption allows us to solve \eqref{eqn:DPP} explicitly when the cost is quadratic. In the rest of this section we fix $\mu = \mathcal{N}(a, A = LL^{\top})$ and $\nu = \mathcal{N}(b, B = MM^{\top})$ where $L, M \in \mathscr{L}_{++}(N)$. By an {\it adapted Wasserstein coupling} $\pi_{\mathrm{AW}}^{\mu, \nu}$ of $(\mu, \nu)$, we mean a coupling $\pi^P$ in the context of Lemma \ref{lem:correlated.coupling} where $P \in \mathscr{P}(N)$ satisfies \eqref{eqn:rho.sign}; it is unique if and only if $(L^{\top}M)_{tt} \neq 0$ for all $t \in [N]$. This notion is handy when solving the DPP.  Observe that since $(L^{\top} M)_{NN} = L_{NN} M_{NN} > 0$, in \eqref{eqn:rho.sign} we always have $\rho_{NN} = +1$.

\begin{proposition}[Time consistency] \label{prop.time.consistent}
	For any $t$ and $x_{\bt}, y_{\bt} \in \mathbb{R}^t$ we have 
	\begin{equation*} 
		(\pi_{\mathrm{AW}}^{\mu, \nu})_{\bt'}(\cdot | x_{\bt}, y_{\bt}) = \pi_{\mathrm{AW}}^{\mu_{\bt'}(\cdot|x_{\bt}), \nu_{\bt'}(\cdot|y_{\bt})},
	\end{equation*}
	in the sense that any version of the left hand side is a version of the right hand side.
\end{proposition}

\begin{lemma}[Conditional distribution under Cholesky decomposition] \label{lem:Cholesky.conditional}
	Let $\mu = \mathcal{N}(a, A = LL^{\top})$ where $L \in \mathscr{L}_{++}(N)$. Given $t$, write $a = (a_{\bt}, a_{\bt'})$ and let
	\begin{equation} \label{eqn:L.block}
		L = \begin{bmatrix} L_{{\bf t}, {\bf t}} & 0 \\ L_{{\bf t}', {\bf t}} & L_{{\bf t}', {\bf t}'} \end{bmatrix}
	\end{equation}
	be the block diagonal representation of the Cholesky matrix $L$. Then
	\begin{equation} \label{eqn:conditional.Choleskdy}
		\mu_{\bt'}(\cdot|x_{\bt})  = \mathcal{N}_{N-t} ( a_{\bt'} + L_{\bt', \bt} L_{\bt, \bt}^{-1}  (x_{\bt} - a_{\bt}) , L_{\bt', \bt'} L_{\bt', \bt'}^{\top}).%
		\footnote{Here (and below) we mean $L_{\bt', \bt'}^{\top} = (L_{\bt', \bt'})^{\top}$. When $t = 0$ the past $x_{{\bf t}}$ is empty. In this case we regard $L_{\bt', \bt} L_{\bt, \bt}^{-1}  (x_{\bt} - a_{\bt}) = 0$ to simplify the presentation. Similarly, $x_{{\bf t}'}$ is empty when $t = N$.}
	\end{equation}
\end{lemma}
\begin{proof}
	Let $X = a + L\epsilon \sim \mu$ where $\epsilon \sim \mathcal{N}_N(0, I)$. In block form, we have
	\begin{equation} \label{eqn:Cholesky.conditional}
		\begin{split}
			X_{\bt} &= a_{\bt} + L_{\bt, \bt} \epsilon_{\bt},\\
			X_{{\bf t}'} &= a_{\bt'} + L_{\bt', \bt} \epsilon_{\bt} + L_{\bt', \bt'} \epsilon_{\bt'} = a_{\bt'} + L_{\bt',\bt} L_{\bt, \bt}^{-1}  (X_{\bt} - a_{\bt}) + L_{\bt', \bt'} \epsilon_{\bt'}.
		\end{split}
	\end{equation}
	Since $\epsilon_{\bt'} \sim \mathcal{N}_{N-t}(0, I)$  and is independent of $X_{
		\bt}$, we immediately obtain \eqref{eqn:conditional.Choleskdy}. Note that the  conditional covariance matrix $L_{\bt', \bt'} L_{\bt', \bt'}^{\top}$ does {\it not} depend on $x_{\bt}$.
\end{proof}

\begin{proof}[Proof of Proposition \ref{prop.time.consistent}]
	Given $t$, consider the block representations
	\[
	L = \begin{bmatrix} L_{\bt, \bt} & 0 \\ L_{\bt', \bt} & L_{\bt', \bt'}\end{bmatrix} \quad \text{and} \quad
	M = \begin{bmatrix} M_{\bt, \bt} & 0 \\ M_{\bt', \bt} & M_{\bt', \bt'}\end{bmatrix}.
	\]
	Under an adapted Wasserstein coupling $\pi_{\mathrm{AW}}^{\mu, \nu}$, write $X = a + L \epsilon^X$ and $Y = b + M \epsilon^Y$ where $(\epsilon^X, \epsilon^Y)$ is distributed as \eqref{eqn:correlated.noise}. From \eqref{eqn:Cholesky.conditional}, we have
	\begin{equation*}
		\begin{split}
			X_{{\bf t}'} = a_{\bt'} + L_{\bt',\bt} L_{\bt, \bt}^{-1}  (X_{\bt} - a_{\bt}) + L_{\bt', \bt'} \epsilon_{\bt'}^X \quad \text{and} \quad
			Y_{{\bf t}'} = b_{\bt'} + M_{\bt',\bt} M_{\bt, \bt}^{-1}  (Y_{\bt} - b_{\bt}) + M_{\bt', \bt'} \epsilon_{\bt'}^Y.
		\end{split}
	\end{equation*}
	Note that $\epsilon^X_{\bt'}, \epsilon^Y_{\bt'}$ are independent of $X_{\bt}, Y_{\bt}$. From \eqref{eqn:correlated.noise}, $(\pi_{\mathrm{AW}}^{\mu, \nu})_{\bt'}(\cdot | x_{\bt}, y_{\bt})$ is the pushforward of
	\[
	\mathcal{N}_{2(N - t)}\left( \begin{bmatrix} 0 \\ 0 \end{bmatrix}, \begin{bmatrix} I & P_{\bt', \bt'} \\ P_{\bt', \bt'} & I \end{bmatrix}\right)
	\]
	under the affine map
	\[
	(\epsilon_{\bt'}^X, \epsilon_{\bt'}^Y) \mapsto 
	(a_{\bt'} + L_{\bt',\bt} L_{\bt, \bt}^{-1}  (x_{\bt} - a_{\bt}) + L_{\bt', \bt'}\epsilon_{\bt'}^X, b_{\bt'} + M_{\bt', \bt} M_{\bt, \bt}^{-1}  (y_{\bt} - b_{\bt}) + M_{\bt', \bt'}\epsilon_{\bt'}^Y).
	\]
	In particular, $(\pi_{\mathrm{AW}}^{\mu, \nu})_{\bt'}(\cdot | x_{\bt}, y_{\bt})$ is a coupling of $(\mu_{\bt'}(\cdot|x_{\bt}), \nu_{\bt'}(\cdot|y_{\bt}))$ and has the form \eqref{eqn:pi.P}. A straightforward computation gives
	\begin{equation*}
		L^{\top} M =
		\left[
		\begin{array}{c|c}
			L_{\bt, \bt}^{\top} M_{\bt, \bt} + L_{\bt', \bt}^{\top} M_{\bt', \bt} & L_{\bt', \bt}^{\top} M_{\bt', \bt'} \\
			\hline
			L_{\bt', \bt'}^{\top} M_{\bt', \bt} & L_{\bt', \bt'}^{\top} M_{\bt', \bt'}
		\end{array}
		\right].
	\end{equation*}
	It follows that $\diag( L^TM)_{\bt'} = \diag (L_{\bt', \bt'}^{\top} M_{\bt', \bt'})$. Hence $P_{\bt', \bt'}$ satisfies \eqref{eqn:rho.sign} for the conditional covariance matrices and we conclude that $(\pi_{\mathrm{AW}}^{\mu, \nu})_{\bt'}(\cdot | x_{\bt}, y_{\bt})$ is an adapted Wasserstein coupling between $\mu_{\bt'}(\cdot|x_{\bt})$ and $\nu_{\bt'}(\cdot|y_{\bt})$.
\end{proof}

\begin{proof}[Proof of Theorem \ref{thm:main}]
	By recentering (i.e., considering $X - a$ and $Y - a$) we may assume $a = b = 0$. From Lemma \ref{lem:Cholesky.conditional}, we have 
	\begin{equation} \label{eqn:conditional.means}
		\begin{split}
			\mathbb{E}_{\mu}[X_{\bt'}|X_{\bt} = x_{\bt}] = L_{\bt', \bt} L_{\bt, \bt}^{-1} x_{\bt} \quad \text{and} \quad
			\mathbb{E}_{\nu}[Y_{\bt'}|Y_{\bt} = y_{\bt}] = M_{\bt', \bt} M_{\bt, \bt}^{-1} y_{\bt}. \\
		\end{split}
	\end{equation}
	
	Let $V_t(x_{1:t}, y_{1:t})$ be the value function in Theorem \ref{thm:DPP}.  We claim that $$V_N(x_{1:N}, y_{1:N}) = \|x_{1:N} - y_{1:N}\|^2$$ (the terminal value) and, for $0 \leq t < N$,
	\begin{equation} \label{eqn:value.function.claim}
		\begin{split}
			&V_t(x_{\bt}, y_{\bt})  \\
			&= \int_{\mathbb{R}^{N-t} \times \mathbb{R}^{N-t}} \|x - y\|^2 \pi_{\mathrm{AW}}^{\mu, \nu} (\dd x_{\bt'} \dd y_{\bt'}|x_{\bt}, y_{\bt})\\
			&= \|x_{\bt} - y_{\bt}\|^2 + \|L_{\bt', \bt} L_{\bt, \bt}^{-1} x_{\bt} - M_{\bt', \bt} M_{\bt, \bt}^{-1} y_{\bt}\|^2
			+ d_{\mathrm{ABW}}^2(L_{\bt', \bt'}L_{\bt', \bt'}^{\top}, M_{\bt', \bt'}M_{\bt', \bt'}^{\top}),
		\end{split}
	\end{equation}
	where the second equality follows from Lemma \ref{lem:correlated.coupling} and the definition of $\pi_{\mathrm{AW}}^{\mu, \nu}$.\footnote{From Lemma \ref{lem:correlated.coupling}, the value of the integral is independent of the version of $\pi_{\mathrm{AW}}^{\mu, \nu}$ used.} If so, the conclusion follows by letting $t = 0$.	
	
	We argue by induction. We use the shorthand $t_{+} = t + 1$ as well as $\bt_{+} = 1:(t + 1)$ and $\bt_{+}' = (t + 2):N$. Suppose \eqref{eqn:value.function.claim} holds for $t + 1$, so that 
	\begin{equation}  \label{eqn:value.function.claim.induction}
		\begin{split}
			V_{t+1}(x_{\bt_{+}}, y_{\bt_{+}}) &=  \|x_{\bt} - y_{\bt}\|^2 + (x_{t + 1} - y_{t + 1})^2 \\
			&\quad + \|L_{\bt_{+}', \bt_{+}} L_{\bt_{+}, \bt_{+}}^{-1} x_{\bt_{+}} - M_{\bt_{+}', \bt_{+}} M_{\bt_{+}, \bt_{+}}^{-1} y_{\bt_{+}}\|^2 \\
			&\quad + d_{\mathrm{ABW}}^2(L_{\bt_{+}', \bt_{+}'}L_{\bt_{+}', \bt_{+}'}^{\top}, M_{\bt_{+}', \bt_{+}'}M_{\bt_{+}', \bt_{+}'}^{\top}).
		\end{split}
	\end{equation}
	Observe that this is a  quadratic function of $x_{t+1}$ and $y_{t+1}$ which appear in the second and third terms. When we integrate it against $\pi_{t+1} \in \Pi(\mu_{t+1}(\cdot|x_{\bt}), \nu_{t+1}(\cdot|y_{\bt}))$, only the integral of the  $x_{t+1}y_{t+1}$ term depends on the choice of the coupling.
	
	To obtain the coefficient of $x_{t+1} y_{t+1}$ we begin by analyzing $L_{\bt_{+}, \bt_{+}}^{-1}$ and $M_{\bt_{+}, \bt_{+}}^{-1}$. Write, in block form,
	\[
	L_{\bt_{+}, \bt_{+}} = 
	\left[
	\begin{array}{c|c}
		L_{\bt, \bt} & 0_{t \times 1} \\
		\hline
		L_{t_{+}, \bt} & L_{t_{+}, t_{+}}
	\end{array}
	\right] \Rightarrow
	L_{\bt_{+}, \bt_{+}}^{-1} = 
	\left[
	\begin{array}{c|c}
		L_{\bt, \bt}^{-1} & 0_{t \times 1} \\
		\hline
		\frac{-1}{L_{t_{+}, t_{+}}} L_{t_{+}, \bt} L_{\bt, \bt}^{-1} & \frac{1}{L_{t_{+}, t_{+}}}
	\end{array}
	\right].
	\]
	It follows that
	\begin{equation} \label{eqn:conditional.mean.coefficients}
		\begin{split}
			L_{\bt_{+}', \bt_{+}} L_{\bt_{+}, \bt_{+}}^{-1} x_{\bt_{+}} &= \begin{bmatrix} L_{\bt_{+}', \bt}  & L_{\bt_{+}', t_{+}} \end{bmatrix}
			\left[
			\begin{array}{c|c}
				L_{\bt, \bt}^{-1} & 0_{t \times 1} \\
				\hline
				\frac{-1}{L_{t_{+}, t_{+}}} L_{t_{+}, \bt} L_{\bt, \bt}^{-1} & \frac{1}{L_{t_{+}, t_{+}}}
			\end{array}
			\right]
			\begin{bmatrix}
				x_{\bt} \\ x_{t_{+}} 
			\end{bmatrix} \\
			&= \left( L_{\bt_{+}', \bt} L_{\bt, \bt}^{-1} - \frac{1}{L_{t_{+}, t_{+}}} L_{\bt_{+}', t_{+}} L_{t_{+}, \bt} L_{\bt, \bt}^{-1} \right)  x_{\bt} + \frac{1}{L_{t_{+}, t_{+}}} L_{\bt_{+}', t_{+}} x_{t_{+}}.
		\end{split}
	\end{equation}
	Similarly, we have 
	\[
	M_{\bt_{+}', \bt_{+}} M_{\bt_{+}, \bt_{+}}^{-1} y_{\bt_{+}} = \left( M_{\bt_{+}', \bt} M_{\bt, \bt}^{-1} - \frac{1}{M_{t_{+}, t_{+}}} M_{\bt_{+}', t_{+}} M_{t_{+}, \bt} M_{\bt, \bt}^{-1} \right) y_{\bt} + \frac{1}{M_{t_{+}, t_{+}}} M_{\bt_{+}', t_{+}} y_{t_{+}}.
	\]
	Combining the above computations, we see that the coefficient of $x_{t+1} y_{t+1}$ in \eqref{eqn:value.function.claim.induction} is $-2\alpha_{t+1}$, where
	\begin{equation*}
		\begin{split}
			\alpha_{t+1} &:= 1 +  \frac{1}{ L_{t_+, t_+}M_{t_+, t_+} } L_{\bt_+', t_+}^{\top} M_{\bt_+', t_+}  = \frac{1}{ L_{t_+, t_+}M_{t_+, t_+} } (L^{\top} M)_{t_+, t_+}.
		\end{split}
	\end{equation*}
	Since $L_{t_+, t_+}M_{t_+, t_+} > 0$, the sign of $\alpha_{t+1}$ is the same as that of $(L^{\top} M)_{t+1, t+1}$ or they are both zero. Completing the squares in \eqref{eqn:value.function.claim.induction}, we have
	\begin{equation} \label{eqn:V.reduced.form}
		V_{t+1}(x_{\bt_{+}}, y_{\bt_{+}}) \ = \|x_{\bt} - y_{\bt}\|^2 + (x_{t+1} - \alpha_{t+1} y_{t+1})^2  + \cdots,
	\end{equation}
	where the omitted terms, when combined, have the form $E_tx_{t+1}^2 + F_ty_{t+1}^2 + G_tx_{t+1} + H_ty_{t+1} + I_t$ for suitable constants which may depend on $x_{1:t}$ and $y_{1:t}$.
	
	Consider now the one-dimensional optimal transport problem
	\begin{equation} \label{eqn:induction.OT}
		\inf_{\pi_{t+1} \in \Pi(\mu_{t+1}(\cdot|x_{\bt}), \nu_{t+1}(\cdot|y_{\bt}))} \int_{\mathbb{R} \times \mathbb{R}} V_{t+1}(x_{\bt_{+}}, y_{\bt_{+}}) \pi(\dd x_{t+1} \dd y_{t+1}),
	\end{equation}
	where $\mu_{t+1}(\cdot|x_{\bt})$ and $\nu_{t+1}(\cdot|y_{\bt})$ are univariate Gaussian distributions. From \eqref{eqn:V.reduced.form}, $\pi_{t+1}$ is optimal for \eqref{eqn:induction.OT} if and only if it is optimal for
	\[
	\inf_{\pi \in \Pi(\mu_{t+1}(\cdot|x_{\bt}), \nu_{t+1}(\cdot|y_{\bt}))} \int_{\mathbb{R} \times \mathbb{R}} (x_{t+1} - \alpha_{t+1} y_{t+1})^2 \pi(\dd x_{t+1} \dd y_{t+1}),
	\]
	which has a quadratic cost. We identify three cases:
	\begin{itemize}
		\item Case 1: $\alpha_{t+1} > 0$ or $(L^{\top} M)_{t+1, t+1} > 0$. The only optimal $\pi_{t+1}$ is the comonotonic coupling between $\mu_{t+1}(\cdot|x_{\bt})$ and $\nu_{t+1}(\cdot|y_{\bt})$.
		\item Case 2: $\alpha_{t+1} < 0$ or $(L^{\top} M)_{t+1, t+1} < 0$. The only optimal $\pi_{t+1}$ is the counter-monotonic coupling between $\mu_{t+1}(\cdot|x_{\bt})$ and $\nu_{t+1}(\cdot|y_{\bt})$.
		\item Case 3: $\alpha_{t+1} = (L^{\top} M)_{t+1, t+1} = 0$. The value of the integral in \eqref{eqn:induction.OT} is independent of the choice of $\pi_{t+1}$. Hence any $\pi_{t+1}$ is optimal.
	\end{itemize}
	By Proposition \ref{prop.time.consistent}, $\pi_{t+1} = \left( \pi_{\mathrm{AW}}^{\mu, \nu} \right)_{t+1}(\cdot|x_{\bt}, y_{\bt})$ is optimal for \eqref{eqn:induction.OT}. Using the induction hypothesis, tower property and Proposition \ref{prop.time.consistent}, we have
	\begin{equation*}
		\begin{split}
			&V_t(x_{\bt}, y_{\bt}) \\
			&= \int V_{t+1}(x_{\bt_{+}}, y_{\bt_{+}}) \left( \pi_{\mathrm{AW}}^{\mu, \nu} \right)_{t+1}(\dd x_{t+1} \dd y_{t+1}|x_{\bt}, y_{\bt}) \\
			&= \int \left( \int \|x - y\|^2 \left( \pi_{\mathrm{AW}}^{\mu, \nu}\right)_{\bt_{+}'}(\dd x_{\bt_{+}'} \dd y_{\bt_{+}'} |x_{\bt_{+}}, y_{\bt_{+}})  \right) \left( \pi_{\mathrm{AW}}^{\mu, \nu} \right)_{t+1}(\dd x_{t+1} \dd y_{t+1}|x_{\bt}, y_{\bt}) \\
			&= \int \|x - y\|^2 \pi_{\mathrm{AW}}^{\mu, \nu} (\dd x_{\bt'} \dd y_{\bt'}|x_{\bt}, y_{\bt}),
		\end{split}
	\end{equation*}
	which is the desired expression and completes the induction. In particular, we may express the value function \eqref{eqn:value.function.claim} as $V_t(x_{\bt}, y_{\bt}) = 
	\|x_{\bt} - y_{\bt}\|^2 + \mathcal{AW}_2^2(\mu_{\bt'}(\cdot|x_{\bt}), \nu_{\bt'}(\cdot|y_{\bt}))$,   with the convention that the last term is $0$ when $t = N$.
\end{proof}

From the proof of Theorem \ref{thm:main} (in particular, the solution to \eqref{eqn:induction.OT}) we obtain the following characterization of the collection of optimal coupling(s).

\begin{corollary}[Characterization of optimal coupling(s)] \label{cor:optimal.couplings}
	Let $\mu = \mathcal{N}(a, A = LL^{\top})$ and $\nu = \mathcal{N}(b, B = MM^{\top})$ be non-degenerate Gaussian distributions on $\mathbb{R}^N$. For $\pi \in \Pi_{bc}(\mu, \nu)$, decompose it as a product of regular conditional distributions as in \eqref{eqn:coupling.decomp}. Then $\pi$ is optimal for $\mathcal{AW}_2(\mu, \nu)$ if and only if for $\pi$-almost all $x_{1:(t-1)}$ and $y_{1:(t-1)}$, $\pi_t(\cdot|x_{1:(t-1)}, y_{1:(t-1)})$ is the comonotonic (resp.~counter-monotonic) coupling between $\mu_t(\cdot|x_{1:(t-1)})$ and $\nu_t(\cdot|x_{1:(t-1)})$ when $(L^TM)_{tt} > 0$ (resp.~$(L^TM)_{tt} < 0$). The optimal coupling is unique if and only if $(L^{\top}M)_{tt} \neq 0$ for all $t$. In this case it is deterministic and is given by
	\begin{equation} \label{eqn:AW.OT.map}
		Y = T_{\mathrm{AW}}^{\mu, \nu} X, \quad T_{\mathrm{AW}}^{\mu, \nu}(x) = b + MPL^{-1}(x - a),
	\end{equation}
	where $P$ is the diagonal matrix with $P_{st} = \delta_{st}  \rho_t$ satisfying \eqref{eqn:rho.sign}. In all cases (any version of) $\pi_{\mathrm{AW}}^{\mu, \nu}$, under which $(X, Y)$ is jointly Gaussian (possibly degenerate), is optimal.
\end{corollary}

\begin{remark} \label{rmk:techniques}  
	\begin{itemize}
		\item[(i)] The proof of Theorem \ref{thm:main} can be extended easily to {\it weighted} square losses of the form $c(x_{1:N}, y_{1:N}) = \sum_{t = 1}^N w_t (x_t - y_t)^2$ where $w_t > 0$. Let $W = (\delta_{st} w_t)$ be diagonal. Then the optimal coefficient $\rho_t$ becomes the sign of $(L^{\top} W M)_{tt}$, and the value of the bicausal OT problem $\inf_{\pi \in \Pi_{bc}(\mu, \nu)} \int_{\mathbb{R}^N \times \mathbb{R}^N} c(x,y) \pi(\dd x \dd  y)$ becomes
		\[
		(a - b)^{\top} W (a - b) + \Tr(L^{\top} W L) + \Tr(M^{\top} W M) - 2 \| \diag (L^{\top} W M) \|_1.
		\]
		\item[(ii)] Our techniques depend heavily on the Gaussian assumption and the linear-quadratic structure of the value function. As shown recently in \cite{AHP24}, this structure extends to the multivariate setting with entropic regularization. While it may be possible to derive the optimal couplings for other distributions such as {\it elliptical distributions} \cite{CHS81}, in this paper we focus on the Gaussian case for concreteness and tractability.\footnote{We thank Ludger R\"{u}schendorf for his suggestion to consider elliptical distributions. Here we note that although the conditional distribution of an elliptical distribution remains elliptical and the conditional mean vector and scale matrix are analogous to those in \eqref{eqn:conditional.Choleskdy}, the conditional covariance matrix depends on the conditioning value unless the distribution is Gaussian (see \cite[page 36]{M82}). So Theorem \ref{thm:main} and Corollary \ref{cor:KR.optimal} do not extend trivially.}
	\end{itemize}
\end{remark}

We end this section with a necessary and sufficient condition for the Knothe-Rosenblatt coupling between non-degenerate Gaussian distributions to be $\mathcal{AW}_2$-optimal.

\begin{corollary} \label{cor:KR.optimal}
	The Knothe-Rosenblatt coupling $\pi_{\mathrm{KR}}^{\mu, \nu}$ is $\mathcal{AW}_2$-optimal  if and only if $(L^{\top}M)_{tt} \geq 0$ for all $t$. In particular, given $A \in \mathscr{S}_{++}(N)$, there exists an open neighborhood $\mathcal{U}$ of $A$ in $\mathscr{S}_{++}(N)$ (under the Euclidean topology) such that when $B \in \mathcal{U}$ then $\pi_{\mathrm{KR}}^{\mu, \nu}$ is optimal for $\mathcal{AW}_2(\mu, \nu)$.
\end{corollary}
\begin{proof}
	The first statement is immediate from Corollary \ref{cor:optimal.couplings}. The second statement follows from the following facts: (i) $(L^{\top} L)_{tt} > 0$ for all $t$; (ii) the map $M \in \mathscr{L}_{++}(N) \mapsto \diag(L^{\top}M) \in \mathbb{R}^N$ is continuous; and (iii) the map $M \mapsto MM^{\top}$ is a diffeomorphism from $\mathscr{L}_{++}(N)$ onto $\mathscr{S}_{++}(N)$.
\end{proof}

\begin{remark} \label{rmk:comonotone}
	It was proved in \cite[Proposition 3.4]{BKR22} that whenever $\mu, \nu \in \mathcal{P}_2(\mathbb{R}^N)$ satisfy a {\it stochastic co-monotonicity condition} the Knothe-Rosenblatt coupling is $\mathcal{AW}_2$-optimal. Even for Gaussian (and Markovian) distributions, our necessary and sufficient condition in Corollary \ref{cor:KR.optimal} is not equivalent to the stochastic co-monotonicity condition. We illustrate this in Example \ref{eg:nonunique}.
\end{remark}

\section{Adapted Wasserstein geometry of Gaussian distributions} \label{sec:AW.distance}
We discuss briefly the adapted Wasserstein geometry of univariate Gaussian processes, and leave a deeper study, in the spirits of \cite{BBEP20, BBP2021, BJL19, T11}, to future research. Let $\mathscr{G}(N) = \{\mathcal{N}(a, A): a \in \mathbb{R}^N, A \in \mathscr{S}_+(N)\}$ be the space of (possibly degenerate) Gaussian distributions on $\mathbb{R}^N$.

It is known that $(\mathcal{P}_2(\mathbb{R}^N), \mathcal{AW}_2)$, unlike $(\mathcal{P}_2(\mathbb{R}^N), \mathcal{W}_2)$, is {\it not} a complete metric space when $N \geq 2$. Its completion can be identified with the space of {\it nested distributions} \cite[Section 4]{BBEP20} or the space of {\it filtered processes} \cite{BBP2021}. Here, we observe that the subspace $(\mathscr{G}(N), \mathcal{AW}_2)$ of Gaussian distributions is still incomplete.  

\begin{proposition}[Incompleteness] \label{prop:incomplete}
	$(\mathscr{G}(N), \mathcal{AW}_2)$ is not complete when $N \geq 2$.
\end{proposition}
\begin{proof}
	We consider the case $N = 2$ which can be extended straightforwardly. For $0 < \theta < \pi = 3.1415\ldots$, let
	\begin{equation} \label{eqn:Cholesky.not.unique}
		L(\theta) = \begin{bmatrix} 0 & 0 \\ \cos \theta & \sin \theta \end{bmatrix} \in \mathscr{L}_+(2) \Rightarrow L(\theta) L(\theta)^{\top} = A := \begin{bmatrix} 0 & 0 \\ 0 & 1 \end{bmatrix} \in \mathscr{S}_+(2).
	\end{equation}
	For $n \geq 1$, let $L_n(\theta) = L(\theta) + \frac{1}{n} I \in \mathscr{L}_{++}(2)$ and $A_n(\theta) = L_n(\theta) L_n(\theta)^{\top} \in \mathscr{S}_{++}(2)$. Let $\mu_n(\theta) = \mathcal{N}_2(0, A_n(\theta))$. Since
	\[
	\mathcal{AW}_2^2(\mu_n(\theta), \mu_m(\theta)) \leq \mathcal{KR}_2^2(\mu_n(\theta), \mu_m(\theta)) = \|L_n(\theta) - L_m(\theta)\|_{\mathrm{F}}^2 = 2 \left(\frac{1}{n} - \frac{1}{m}\right)^2, 
	\]
	the sequence $\{\mu_n(\theta)\}_{n \geq 1}$ is Cauchy in $(\mathcal{P}_2(\mathbb{R}^N), \mathcal{AW}_2)$. Since $\mathcal{W}_2 \leq \mathcal{AW}_2$, for all $\theta$ the only possible limit $\mu$ is the limit based on weak convergence, i.e., $\mu = \mathcal{N}(0, A)$. But unless $\theta = \theta'$ we have
	\[
	\lim_{n \rightarrow \infty} \mathcal{AW}_2^2(\mu_n(\theta), \mu_n(\theta')) = 2 - 2(|\cos \theta \cos \theta'| + \sin \theta \sin \theta') > 0,
	\]
	which contradicts the existence of a limit.
\end{proof}

We explain the probabilistic intuition of the construction in the proof of Proposition \ref{prop:incomplete}. Write $X = L(\theta) \epsilon$, so that $X_1 = 0$ and $X_2 = (\cos \theta) \epsilon_1 + (\sin \theta) \epsilon_2$. If we use the natural filtration $\mathcal{F}^X$ of $X$, $X_1 = 0$ contains no information about $X_2$. However, only $\epsilon_2$ is independent of $\mathcal{F}_1^{\epsilon}$ and the ``angle'' $\theta$ controls the relative exposure. Although $\mathcal{F}^X = \mathcal{F}^{\epsilon}$ when $X = L_n(\theta) \epsilon$, the equality between the two filtrations is not preserved in the limit.

Next we consider geodesics.\footnote{A curve $(\gamma_t)_{0 \leq t \leq 1}$ is said to be a (constant-speed) {\it geodesic} on a metric space $(S, d)$ if $d(\gamma_s, \gamma_t) = |s - t| d(\gamma_0, \gamma_1)$ for all $s, t \in [0, 1]$.} It is well known that the unique geodesic $(\mu_t)_{0 \leq t \leq 1}$ between non-degenerate $\mu_0 = \mathcal{N}(a_0, A_0)$ and $\mu_1 = \mathcal{N}(a_1, A_1)$ in $(\mathcal{P}_2(\mathbb{R}^N), \mathcal{W}_2)$ is given by McCann's {\it displacement interpolation} \cite{M97}: $\mu_t = \mathcal{N}(a_t, A_t)$, where $a_t=(1 - t)a_0 + ta_1$,
\begin{equation} \label{eqn:matrix.geodeisc}
	A_t = T_t A_0 T_t^{\top}, \quad  T_t = (1 - I)t + t T,
\end{equation}
and $T = A_0^{-\frac{1}{2}} \big(  A_0^{\frac{1}{2}} A_1  A_0^{\frac{1}{2}} \big)^{\frac{1}{2}} A_0^{-\frac{1}{2}} $ is the Jacobian matrix in \eqref{eqn:W.OT.map}. Replacing $T$ in \eqref{eqn:matrix.geodeisc} by $L_1L_0^{-1}$ (where $L_i$ is the Cholesky matrix of $A_i$) gives the geodesic in $(\mathscr{G}(N), \mathcal{KR}_2)$, and amounts to interpolating linearly the Cholesky matrices: $L_t = (1 - t)L_0 + tL_1$.

\begin{proposition} \label{prop:ABW.geodesics}
	Let $A_0, A_1 \in \mathscr{S}_{++}(N)$. For $0 \leq t \leq 1$, Define $A_t$ by \eqref{eqn:matrix.geodeisc} where $T = L_1 P L_0^{-1}$ as in \eqref{eqn:AW.OT.map}, $P$ satisfies \eqref{eqn:rho.sign} and $\rho_{t} \in \{-1, +1\}$ when $(L_0^{\top} L_1)_{tt} = 0$. Then $\mu_t = \mathcal{N}(a_t, A_t)$, where $a_t = (1 - t)a_0 + ta_1$, is a geodesic in $(\mathscr{G}(N), \mathcal{AW}_2)$. 
\end{proposition}
\begin{proof}
	On some probability space, let $X_0 \sim \mathcal{N}(a_0, A_0)$. For $0 \leq t \leq 1$, let $X_t = a_t + T_t(X_0 - a_0)$. Since $T_t$ is lower triangular and invertible, $(X_s, X_t)$ is a bicausal coupling of $(\mu_s, \mu_t)$ for $0 \leq s, t \leq 1$. From Corollary \ref{cor:optimal.couplings}, $(X_0, X_1)$ is $\mathcal{AW}_2$-optimal for $(\mu_0, \mu_1)$. From \eqref{eqn:matrix.geodeisc} we have, for $0 \leq s, t \leq 1$,
	\[
	\mathcal{AW}_2(\mu_s, \mu_t) \leq \mathbb{E}[\|X_s - X_t\|^2]^{\frac{1}{2}} = |s - t| \mathbb{E}[\|X_0 - X_1\|^2]^{\frac{1}{2}} = |s - t| \mathcal{AW}_2(\mu_0, \mu_1).
	\]
	Since $s$ and $t$ are arbitrary, the triangle inequality implies that equality holds.
\end{proof}

We note that in Proposition \ref{prop:ABW.geodesics} the $\mathcal{AW}_2$-geodesic is not necessarily unique when $L_0^{\top}L_1$ contains a zero entry. In fact, when $L_0^{\top}L_1$ contains a negative entry the interpolant $A_t$ may become degenerate, see Example \ref{eg:compare}.

It is well-known that the Bures-Wasserstein distance $d_{\mathrm{BW}}$ is a Riemannian distance on $\mathscr{S}_{++}(N)$ \cite{T11}; the Riemannian metric is Otto's metric \cite{O01} restricted to the space of centered Gaussian distributions. From Corollary \ref{cor:KR.geometry}, the Knothe-Rosenblatt distance $d_{\mathrm{KR}}$ is Euclidean when expressed in terms of the Cholesky matrix. From \eqref{eqn:ABW.Gaussian} and \eqref{eqn:d.KR}, $d_{\mathrm{ABW}}$ and $d_{\mathrm{KR}}$ are related by
\begin{equation} \label{eqn:ABW.expression2}
	d_{\mathrm{ABW}}^2(A, B) = d_{\mathrm{KR}}^2(A, B) - 4 \sum_{t: (L^{\top}M)_{tt} < 0} |(L^{\top}M)_{tt}|.
\end{equation}

\begin{proposition}
	The adapted Bures-Wasserstein distance $d_{\mathrm{ABW}}$ is not a Riemannian distance on $\mathscr{S}_{++}(N)$.
\end{proposition}
\begin{proof}
	We argue by contradiction. Suppose $d_{\mathrm{ABW}}$ is distance induced by a Riemannian metric on $\mathscr{S}_{++}(N)$. The metric tensor at $A \in \mathscr{S}_{++}(N)$ can be recovered by the distances $d_{\mathrm{ABW}}(A, B)$ where $B$ ranges over a neighborhood of $A$ \cite[Example 9]{WY22}. From Corollary \ref{cor:KR.optimal}, we have $d_{\mathrm{ABW}}(A, B) = d_{\mathrm{KR}}(A, B)$ when $A$ and $B$ are sufficiently close. This implies that the metric tensor must be that of $d_{\mathrm{KR}}$ (which is Euclidean under the Cholesky coordinate system $A \mapsto L$). This is a contradiction since $d_{\mathrm{ABW}} \not\equiv d_{\mathrm{KR}}$ from \eqref{eqn:ABW.expression2}.
\end{proof}

\section{Examples} \label{sec:examples}

In the following examples we let $\mu = \mathcal{N}(0, A)$ and $\nu = \mathcal{N}(0, B)$ be non-degenerate.

\begin{example}
	Suppose $A$ is diagonal. Then $L = A^{\frac{1}{2}}$ is also diagonal. For any $B$, we have $(L^{\top}M)_{tt} = L_{tt} M_{tt} > 0$ for all $t$. By Corollary \ref{cor:KR.optimal}, the Knothe-Rosenblat coupling is $\mathcal{AW}_2$-optimal. This is a special case of \cite[Theorem 2.9]{BBYZ2017}.
\end{example}

\begin{example} \label{eg:compare}
	Consider
	\begin{equation} \label{eqn:compare}
		A = \begin{bmatrix} 1 & 2 \\ 2 & 5 \end{bmatrix} \quad \text{and} \quad B = \begin{bmatrix} 1 & -2 \\ -2 & 5 \end{bmatrix}.
	\end{equation}
	The Wasserstein transport $Y = T_{\mathrm{W}}^{\mu, \nu} X$ is visualized in the top row of Figure \ref{fig:compare}. Since $\diag(L^{\top}M) = (-3, 1)$ contains a negative entry, the Knothe-Rosenblatt coupling is not $\mathcal{AW}_2$-optimal. The second row of Figure  \ref{fig:compare} visualizes the transports $Y =  T_{\mathrm{KR}}^{\mu, \nu}X$ and  $Y = T_{\mathrm{AW}}^{\mu, \nu} X$ (the later is unique and deterministic by Corollary \ref{cor:optimal.couplings}). Here we have $\mathcal{W}_2(\mu, \nu) \approx 1.75$, $\mathcal{KR}_2(\mu, \nu) = 4$ and $\mathcal{AW}_2(\mu, \nu) = 2$ (which is consistent with \eqref{eqn:ABW.expression2}). Note that when $N = 2$ all distributions are Markovian. So the Knothe-Rosenblatt coupling is not necessarily optimal even for Gauss-Markov processes. We also visualize, in Figure \ref{fig:interpolation}, the interpolations \eqref{eqn:matrix.geodeisc} based on $T_{\mathrm{W}}^{\mu, \nu}$, $T_{\mathrm{KR}}^{\mu, \nu}$ and $T_{\mathrm{AW}}^{\mu, \nu}$.

	\begin{figure}[t!]
		\centering
		\includegraphics[scale = 0.35]{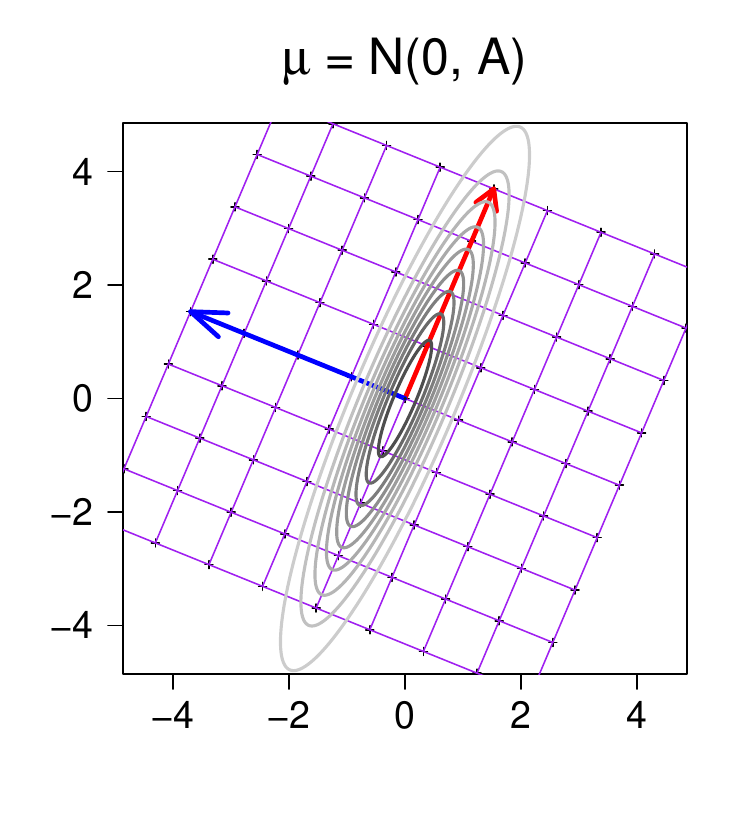}
		\includegraphics[scale = 0.35]{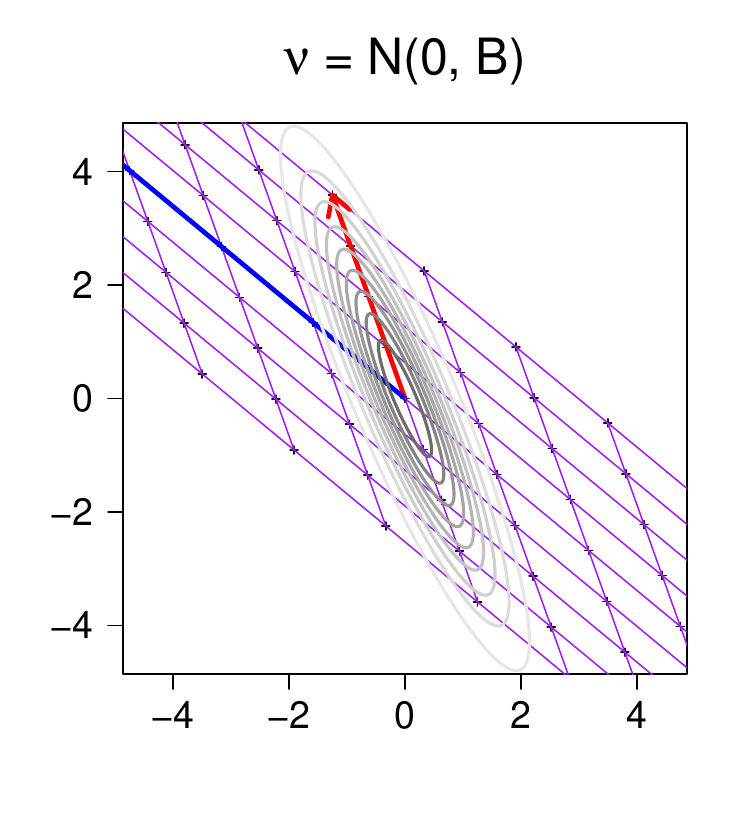}
		\includegraphics[scale = 0.35]{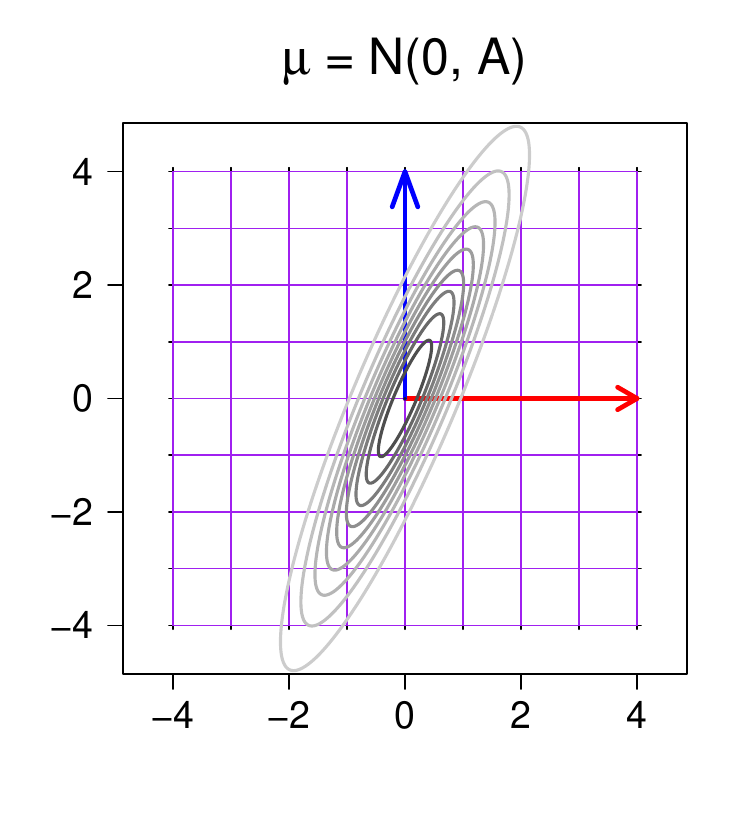}
		\hspace{-0.62cm}
		\includegraphics[scale = 0.35]{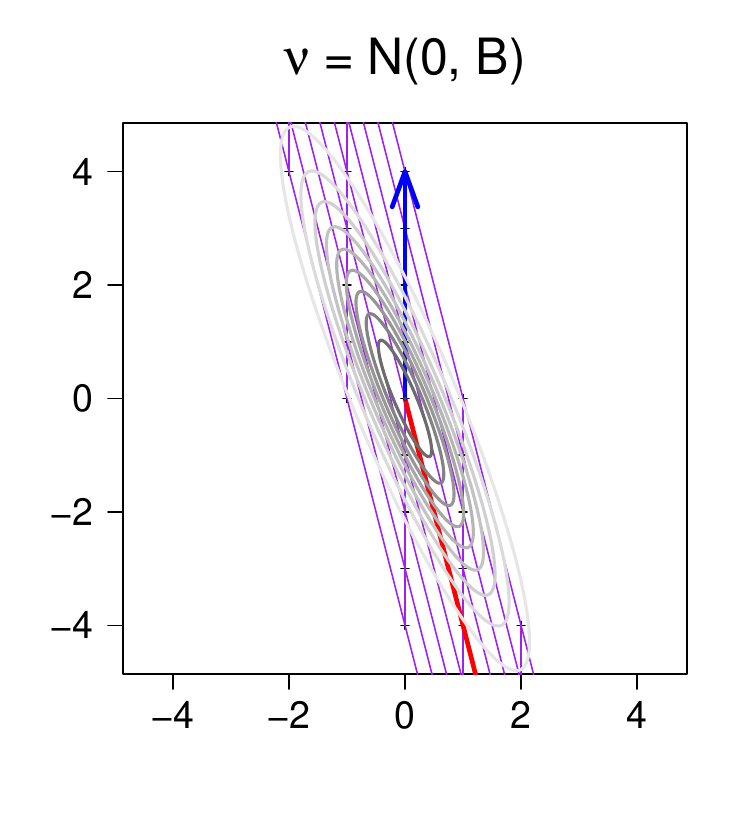}
		\hspace{-0.62cm}
		\includegraphics[scale = 0.35]{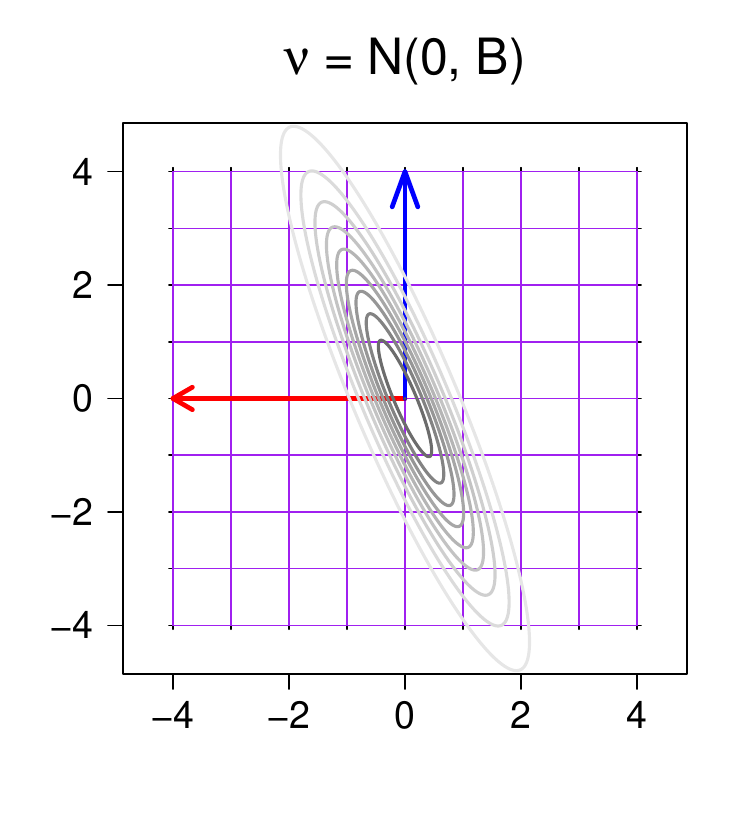}
		\vspace{-0.5cm}
		\caption{Transports between $\mu = \mathcal{N}(0, A)$ and $\nu = \mathcal{N}(0, B)$ where $A$ and $B$ are given by \eqref{eqn:compare}. Each distribution is visualized by the elliptical contours of its density. Top row: The $\mathcal{W}_2$-transport. Here the arrows and the grid correspond to the eigenvectors of $A$. The image of the grid under $T_{\mathrm{W}}^{\mu, \nu}$ is shown on the right. Second row: Now the arrows and the grid correspond to the standard Euclidean basis $\{\mathbf{e}_1, \mathbf{e}_2\}$ in the source space. The middle plot shows the transport under $T_{\mathrm{KR}}^{\mu, \nu}$ and the right plot shows the transport under $T_{\mathrm{AW}}^{\mu, \nu}$. Since $T_{\mathrm{KR}}^{\mu, \nu}$ is comonotonic at each time, $T_{\mathrm{KR}}^{\mu, \nu} \mathbf{e}_1$ points towards the right and $T_{\mathrm{KR}}^{\mu, \nu}\mathbf{e}_2$ points upward. Here only $T_{\mathrm{ABW}}^{\mu, \nu}\mathbf{e}_2$ points upward.} \label{fig:compare}
	\end{figure}
	
	\begin{figure}[t!]
		\centering
		\includegraphics[scale = 0.42]{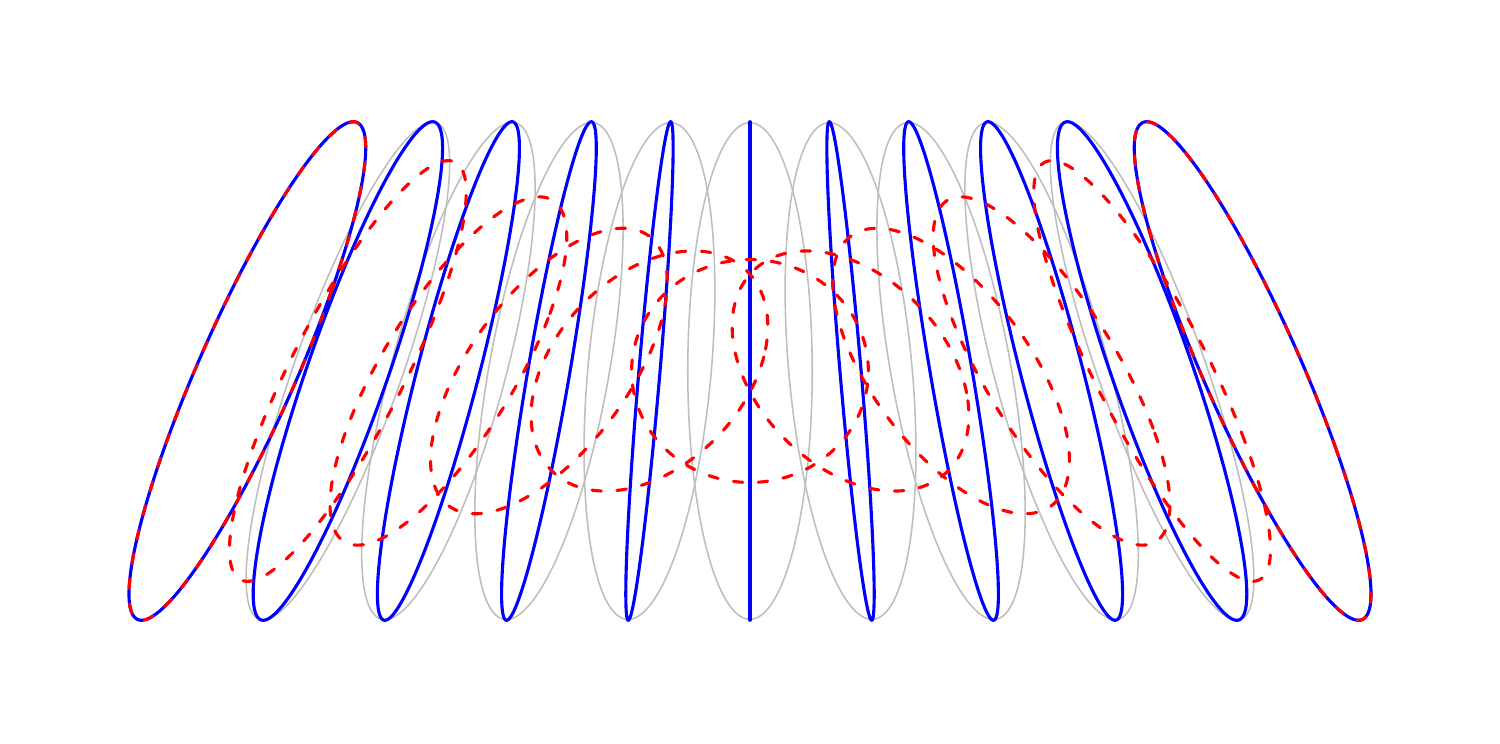}
		\vspace{-0.6cm}
		\caption{Three interpolations $(A_t)_{0 \leq t \leq 1}$ between $A_0 = A$ and $A_1 = B$ as in \eqref{eqn:compare}. For visualization purposes we also include a uniform translation. The thin (grey) ellipses visualize McCann's displacement interpolation. The thick dashed (red) ones show the Knothe-Rosenblatt interpolation where $T = L_1 L_0^{-1}$. The thick solid (blue) ones show the interpolation with $T = L_1 P L_0^{-1}$; in this case  $A_{\frac{1}{2}}$ is degenerate.} \label{fig:interpolation}
	\end{figure}
\end{example}

Next we give an example where the optimal $\mathcal{AW}_2$-coupling is not unique.

\begin{example} \label{eg:nonunique}
	Consider
	\begin{equation} \label{eqn:nonunique}
		L = \begin{bmatrix} 1 & 0 \\ 1 & 1 \end{bmatrix} \Rightarrow A = \begin{bmatrix} 1 & 1 \\ 1 & 2 \end{bmatrix} \quad \text{and} \quad M = \begin{bmatrix} 1 & 0 \\ -1 & 1 \end{bmatrix} \Rightarrow  B = \begin{bmatrix} 1 & -1 \\ -1 & 2 \end{bmatrix}.
	\end{equation}
	Then $\diag(L^{\top}M) = (0, 1)$. Thus any optimal $\mathcal{AW}_2$-coupling is comonotonic at time $2$ but is otherwise unrestricted. To see intuitively why, use \eqref{eqn:nonunique} to write 
	\begin{equation} \label{eqn:last.example}
		X_2 = X_1 + \epsilon_2, \quad Y_2 = -Y_1 + \epsilon_2,
	\end{equation}
	where $\epsilon_2 = \epsilon_2^X = \epsilon_2^Y \sim \mathcal{N}(0, 1)$ is the common noise at time $2$. Since
	\[
	\|X - Y\|^2 = (X_1 - Y_1)^2 + (X_2 - Y_2)^2 = (X_1 - Y_1)^2 + (X_1 + Y_1)^2,
	\]
	the cross terms cancel out. Thus the transport cost does not depend on the coupling at time $1$. 
	
	Since $\diag(L^{\top}M)$ has non-negative entries, the Knothe-Rosenblatt coupling is optimal. From \eqref{eqn:last.example}, we see that $\mu$ is stochastically increasing but $\nu$ is stochastically decreasing, so $\mu$ and $\nu$ are not stochastically co-monotone (we refer the reader to \cite[Definition 3.3]{BKR22} for definitions). In particular, stochastic co-monotonicity is not necessary for the Knothe-Rosenblatt coupling to be optimal, even if we restrict to Gauss-Markov processes.        
\end{example}

\section{Conclusion} \label{sec:conclusion}
In this paper we derived explicitly the adapted Wasserstein distance between non-degenerate Gaussian distributions and characterized the optimal couplings. We also studied some properties of the resulting adapted Bures-Wasserstein distance between covariance matrices. Our results, as well as the extensions in \cite{AHP24}, suggest many directions for future research. Apart from a deeper study of the adapted Bures-Wasserstein distance (including the degenerate case), we may consider the {\it adapted Wasserstein barycenter} introduced in \cite{AKP24}. Motivated by the Wasserstein setting \cite{AC11}, it seems natural to conjecture that when the marginals are Gaussian, there exists an adapted Wasserstein barycenter which is Gaussian. It is also of interest to use the explicit formula of $d_{\mathrm{ABW}}$ to study, in the spirit of \cite{H23}, distributionally robust optimization problems, as well as filtering, projection and stochastic control under adapted Wasserstein loss. Finally, our techniques may be useful for studying the adapted Wasserstein distance between continuous-time Gaussian processes which are not necessarily semimartingales.

\section*{Acknowledgement}
The research of T.-K.~L.~Wong is partially supported by an NSERC Discovery Grant (RGPIN-2019-04419). We presented some preliminary results at the workshop Optimal Transport and Distributional Robustness hosted at the Banff International Research Station. We thank the organizers and participants for their helpful comments.

\bibliographystyle{abbrv}
\bibliography{transport.bib}
\end{document}